\documentclass[11pt]{amsart}
\usepackage{amsmath}
\usepackage{amssymb}
\usepackage{amsfonts}
\usepackage{latexsym}
\usepackage{amscd}
\usepackage[mathscr]{euscript}
\usepackage{xy} \xyoption{all}

\newcommand{\SpecZ}{\textrm{Spec }\mathbb{Z}}
\newcommand{\Specfun}{\textrm{Spec }\mathbb{F}_1}

\newcommand{\uloopr}[1]{\ar@'{@+{[0,0]+(-4,5)}@+{[0,0]+(0,10)}@+{[0,0] +(4,5)}}^{#1}}
\newcommand{\uloopd}[1]{\ar@'{@+{[0,0]+(5,4)}@+{[0,0]+(10,0)}@+{[0,0]+ (5,-4)}}^{#1}}
\newcommand{\dloopr}[1]{\ar@'{@+{[0,0]+(-4,-5)}@+{[0,0]+(0,-10)}@+{[0, 0]+(4,-5)}}_{#1}}
\newcommand{\dloopd}[1]{\ar@'{@+{[0,0]+(-5,4)}@+{[0,0]+(-10,0)}@+{[0,0 ]+(-5,-4)}}_{#1}}

\newcommand{\luloop}[1]{\ar@'{@+{[0,0]+(-8,2)}@+{[0,0]+(-10,10)}@+{[0, 0]+(2,2)}}^{#1}}

\newtheorem{lem}{Lemma}[section]
\newtheorem{corol}[lem]{Corollary}
\newtheorem{theor}[lem]{Theorem}
\newtheorem{prop}[lem]{Proposition}
\theoremstyle{definition}
\newtheorem{defi}[lem]{Definition}

\newtheorem{rema}[lem]{Remark}

\newcommand{\nc}{\newcommand}
\nc{\fin}{\mathrm{fin}}
\nc{\eilenberg}{\mathrm{H}}
\nc{\id}{\mathrm{id}}
\nc{\op}{\mathrm{op}}
\nc{\supp}{\mathrm{Supp}}
\nc{\supph}{\mathrm{supph}}
\nc{\set}{\mathrm{set}}
\nc{\sing}{\mathrm{sing}}
\nc{\pr}{\mathrm{Pr}}
\nc{\psets}{\mathrm{pSet}}
\nc{\defin}{\mathrm{def}}
\nc{\spec}{\mathrm{Spec}}
\nc{\Hom}{\mathrm{Hom}}
\nc{\spc}{\mathrm{Spc}}
\nc{\ch}{\mathrm{Ch}}
\nc{\cl}{\mathrm{Cl}}
\nc{\im}{\mathrm{im}}
\nc{\htp}{\mathrm{h}}
\nc{\fgt}{\mathrm{fgt}}
\nc{\perf}{\mathrm{perf}}
\nc{\ouv}{\mathrm{Ouv}}
\begin{document}

\title[Reconstructing the spectrum of ${\mathbb F}_1$ from the stable homotopy category]{Reconstructing the spectrum of ${\mathbb F}_1$ \\from the stable homotopy category}%
\author{Stella Anevski}
\address{} \email{}

\thanks{}
\date{}
\begin{abstract}
The finite stable homotopy category ${\mathcal S}_0$ has been suggested as a candidate for a category of perfect complexes over the monoid scheme $\Specfun$. We apply a reconstruction theorem from algebraic geometry to ${\mathcal S}_0$, and show that one recovers the one point topological space. 
We also classify filtering subsets of the set of principal thick subcategories of ${\mathcal S}_0$, and of its $p$-local versions. This is motivated by a result saying that the analogous classification for the category of perfect complexes over an affine scheme provides topological information.
\end{abstract}
\maketitle
\section*{Introduction}
\noindent
Let $X$ be a scheme.  Following \cite{sga6}, one associates a category $D^{\perf}(X)$ to $X$. The category $D^{\perf}(X)$ is very close to the category of algebraic vector bundles over $X$, and is well-known for encoding fundamental invariants. Recent research shows that for a large class of schemes, the category $D^{\perf}(X)$ is just as good as $X$, provided that $D^{\perf}(X)$ is endowed with the structure of a tensor triangulated category.\\
\begin{theor} {\em (P. Balmer \cite{presh} and \cite{spc})}\label{balmer2}
 $$\begin{array}{rl} 
 (i) &\mbox{Two reduced noetherian schemes $X$ and $Y$ are isomorphic if and only if $D^{\perf}(X)$ }\\
 &\mbox{and $D^{\perf}(Y)$ are equivalent as tensor triangulated categories.}\\
 (ii) &\mbox{A noetherian scheme $X$ can be reconstructed up to isomorphism from the tensor}\\
 &\mbox{triangulated category $D^{\perf}(X)$.}
\end{array}\,$$
\end{theor}
 \noindent
The reconstruction method used in $(ii)$ can be applied to any tensor triangulated category, producing a locally ringed space. Hence one could adopt the philosophy that tensor triangulated categories represent some kind of (generalized) algebro-geometric objects.\\
 \\
A natural question to pose is what the reconstruction procedure yields when applied to a tensor triangulated category which do not a priori arise as $D^{\perf}(X)$ for a scheme $X$. In this article, we apply it to the finite stable homotopy category ${\mathcal S}_0$, and to its $p$-local versions ${\mathcal S}^{p}_0$. Our reasons for choosing these particular categories include:
  $$\begin{array}{rl} 
(i) &\mbox{The category ${\mathcal S}_0$ has been suggested as an analog of $D^{\perf}(X)$ for the monoidscheme}\\
&\mbox{$\Specfun$ (cf. Remark \ref{f1}). With this in mind, one would hope to reconstruct the one}\\
&\mbox{point topological space from ${\mathcal S}_0$.} \\
 \\
 (ii) &\mbox{The category $D^{\perf}(X)$ can be regarded as a homotopy category of complexes (cf.}\\
 &\mbox{section $1.1$). Inspired by Theorem \ref{balmer2}, we think of noetherian algebraic geometry}\\
 &\mbox{as homotopy theory in categories of complexes.}
\end{array}\,$$
  \\
 

\newpage
\noindent
We use P. Balmer's construction of the {\em spectrum} of a tensor triangulated category. The spectrum is a topological space with points corresponding to a certain type of triangulated full subcategories, which will be called {\em primes}. 
Our main results are contained in Theorems \ref{stolt} and \ref{stolt2}. \\
 \\
 Theorem \ref{stolt} classifies the primes of ${\mathcal S}_0$ and ${\mathcal S}^{p}_0$. In particular, part $(iv)$ reveals that the topological space reconstructed from ${\mathcal S}_0$  has just one point, which is indeed consistent with viewing the finite stable homotopy category as an analog of $D^{\perf}(X)$ for the monoidscheme $\Specfun$. \\
  \\
  Theorem \ref{stolt2} classifies filtering subsets of a certain set of subcategories of ${\mathcal S}_0$ and ${\mathcal S}^{p}_0$. In particular it answers a question posed by M. B\"okstedt in the appendix of \cite{neeman}. For an affine (not necessarily noetherian) scheme $\spec R$, the analogous classification for $D^{\perf}(\spec R)$ gives information about the closed subsets of $\spec R$. \\
 \\
The exposition is organized as follows: Section $1$ begins with a brief review of tensor triangulated categories. We define the category $D^{\perf}(X)$, of perfect complexes over a scheme $X$, and describe how one can recover the underlying topological space of $X$ from the tensor triangulated structure on $D^{\perf}(X)$, provided that $X$ is noetherian. In section $2$, we summarize the construction of the spectrum of a tensor triangulated category, and state some of its properties. Section $3$ introduces the finite stable homotopy category, with emphasis on its structural properties. In particular, we state the thick subcategory theorem by M. J. Hopkins and J. H. Smith, classifying thick subcategories of the $p$-local finite stable homotopy category.  The section ends with a global version of this theorem, which is due to S. M. Chebolu. Sections $4$ and $5$ contain statements and proofs of our main results.\\
 \\
\section{Tensor triangulated categories}
\noindent
Recall that a {\em triangulated category} is an additive category ${\mathcal K}$, equipped with an autoequivalence $\Sigma: {\mathcal K}\to {\mathcal K}$, and a class of {\em exact triangles}, i.e. diagrams in ${\mathcal K}$ of the form
 $$\begin{array}{rl} 
P\to Q\to R\to \Sigma (P).
\end{array}\,$$
This class is subject to certain axioms, whose raison d'\^etre is to force exact triangles to encode homological information about ${\mathcal K}$ (cf. \cite[II:$1.1.1$]{verdier}). An additive functor between triangulated categories is called {\em exact} if it preserves exact triangles and commutes with the autoequivalences up to isomorphism.
\begin{defi} \label{tt}
A {\em tensor triangulated category} is a pair $({\mathcal K}, \otimes)$, consisting of a triangulated category ${\mathcal K}$, and a bifunctor $\otimes:{\mathcal K}\times{\mathcal K}\to {\mathcal K}$ which is exact in each variable.
\end{defi}
\noindent
We now give names to the subcategories that will be associated with closed subsets and points under the reconstruction procedure.
\begin{defi}
Let ${\mathcal K}$ be a tensor triangulated category. A full triangulated subcategory ${\mathcal A}$ of ${\mathcal K}$ is said to be
$$\begin{array}{rl} 
- &\mbox{{\em thick}, if it is closed under retracts, i.e. $P\oplus Q\in {\mathcal A}$ forces $P\in {\mathcal A}$ and $Q\in{\mathcal A}$,}\\
- &\mbox{{\em $\otimes$-thick}, if it is thick and $Q\in {\mathcal A}$ and $P\in {\mathcal K}$ implies $P\otimes Q\in {\mathcal A}$,}\\
- &\mbox{a {\em prime of ${\mathcal K}$}, if ${\mathcal A}\neq {\mathcal K}$, ${\mathcal A}$ is $\otimes$-thick, and $P\otimes Q\in {\mathcal A}$ implies $P\in {\mathcal A}$ or $Q\in {\mathcal A}$.}
\end{array}\,$$
\end{defi}
 \newpage
\subsection{Perfect complexes and reconstruction of schemes} \label{perfect}
Let $X$ be a scheme. A {\em strict perfect complex} on $X$ is a bounded complex of locally free ${\mathcal O}_X$-modules of finite type. A complex $P$ of ${\mathcal O}_X$-modules is {\em perfect} if any point has a neighborhood $U$, such that the restriction of $P$ to $U$ is quasi-isomorphic to some strict perfect complex on $U$.\\
 \\
 The category ${\mathcal O}_{X}$-mod of sheaves of ${\mathcal O}_X$-modules is abelian, and so its derived category $D({\mathcal O}_X-mod)$ is triangulated. The derived tensor product $\otimes^{\mathbb L}$ induces a monoidal structure on  $D({\mathcal O}_X$-mod$)$ which is furthermore compatible with the triangulation in the sense of Definition \ref{tt}. This tensor triangulated structure induces a tensor triangulated structure on the full subcategory $D^{\perf}(X)$ of perfect complexes in $D({\mathcal O}_X$-mod$)$.  \\
  \\
 In the next section, we summarize P. Balmer's procedure for reconstructing a noetherian scheme $X$ from the category $D^{\perf}(X)$. To get a feeling for what is abstracted there, we describe the correspondence between closed subsets of a scheme and thick subcategories of its category of perfect complexes. It relies on the following notion of {\em support} of a complex.\\
  \begin{defi}
Let $X$ be a scheme. The {\em homological support} of a perfect complex $P\in D^{\perf}(X)$ is the set
$$\begin{array}{rl} 
\supph(P):=\{x\in X  |  P_x\ncong 0 \}.
\end{array}\,$$
\end{defi}
\noindent
 Recall that a subset $Y$ of a scheme is said to be {\em specialization closed} if
 $$\begin{array}{rl} 
y\in Y \Rightarrow \overline{{\{y\}}}\subset Y.
\end{array}\,$$
For a noetherian scheme $X$, specialization closed subsets are in bijection with $\otimes$-thick subcategories of $D^{\perf}(X)$ (cf. \cite{thomason}). Under this bijection, a thick subcategory ${\mathcal A}$ corresponds to the set $\bigcup_{P\in {\mathcal A}}\supph(P)$, i.e. the points $x\in X$ for which there exists a complex in ${\mathcal A}$ with non-vanishing homology at $x$. In the other direction, a specialization closed subset $Y$ corresponds to the full triangulated subcategory consisting of complexes $P$ such that $\supph(P)\subset Y$, i.e. complexes acyclic off $Y$.\\
 \\
 \section{The spectrum of a tensor triangulated category}
\noindent
Following \cite{spc}, we construct a functor $\spc(-)$ from the category of tensor triangulated categories to the category of topological spaces. First, we define it on objects. 
 \begin{defi} \label{thespectrum}
 The {\em spectrum} of a tensor triangulated category $({\mathcal K},\otimes)$ is the set of primes of ${\mathcal K}$:
 $$\begin{array}{rl} 
 \spc({\mathcal K}):=\{{\mathcal P}  |  {\mathcal P} \mbox{ is a prime of ${\mathcal K}$}\}.
\end{array}\,$$
We endow the spectrum with a topology by declaring closed sets to be those of the form
$$\begin{array}{rl} 
Z({\mathcal F}):=\{{\mathcal P}\in \spc({\mathcal K})  |  {\mathcal F}\cap {\mathcal P}=\emptyset\},
\end{array}\,$$
for some family of objects ${\mathcal F}\subset {\mathcal K}$.
\end{defi}
\noindent
It is easily checked that the sets $Z({\mathcal F})$ actually satisfy the axioms for a collection of closed sets of a topology.\\ 
 \\
\begin{prop} {\em (Functoriality, \cite[Proposition $3.6$]{spc})}\\
Given a monoidal exact functor ${\mathcal K}\xrightarrow{F} {\mathcal L}$ between tensor triangulated categories, the map
$$\begin{array}{rl} 
 \spc({\mathcal L})&\xrightarrow{\spc (F)} \spc({\mathcal K}),\\
 Q &\mapsto F^{-1}(Q)
\end{array}\,$$
is well-defined and continuous. 
\end{prop}
   \noindent
     \begin{defi}
  Let $({\mathcal K},\otimes)$ be a tensor triangulated category.  The {\em support} of an object $P\in {\mathcal K}$ is the closed set
  $$\begin{array}{rl} 
\supp(P):=Z(\{P\})=\{{\mathcal P}\in\spc({\mathcal K})  |  P\not\in{\mathcal P}\}.
\end{array}\,$$
\end{defi}
\noindent
One checks that the functor $\spc(-)$ satisfies $(\spc (F))^{-1}(\supp (P))= \supp (F(P))$ for all objects $P\in {\mathcal K}$.\\
 \\
   \begin{defi}
   A {\em support data} on a tensor triangulated category $({\mathcal K}, \otimes)$ is a pair $(X,\sigma)$, where $X$ is a topological space, and $\sigma$ is an assignment of a closed subset for each object of ${\mathcal K}$ satisfying the following properties:
   $$\begin{array}{rl} 
- &\mbox{$\sigma (0)=\emptyset$,  and $\sigma(1)=X$,}\\
- &\mbox{$\sigma (P\oplus Q)=\sigma (P)\cup \sigma(Q)$,}\\
- &\mbox{$\sigma (\Sigma P)=\sigma (P)$ for the autoequivalence $\Sigma$,}\\
- &\mbox{$\sigma (P)\subset \sigma (Q)\cup \sigma (R)$ for any exact triangle $P\to Q\to R \to \Sigma P$,}\\
- &\mbox{$\sigma (P\otimes Q)=\sigma (P)\cap \sigma (Q)$.}
\end{array}\,$$
A morphism $(X,\sigma)\to (Y,\tau)$ of support data on $({\mathcal K}, \otimes)$ is a continuous map $f:X\to Y$ such that $\sigma(P)=f^{-1}(\tau(P))$ for all objects $P\in {\mathcal K}$.
   \end{defi}
\noindent
The support data $(\spc({\mathcal K}), \supp)$ is the final support data on ${\mathcal K}$ in the following sense.
   \begin{prop} {\em (Universality, \cite[Theorem 3.2]{spc})}\\
   Let ${\mathcal K}$ be a tensor triangulated category. For any support data $(X,\sigma)$ on ${\mathcal K}$, there is a unique continuous map $f:X\to \spc({\mathcal K})$ such that $\sigma(P)=f^{-1}(\supp(P))$ for all $P\in {\mathcal K}$, namely
   $$\begin{array}{rl} 
X&\xrightarrow{f} \spc({\mathcal K})\\
x&\mapsto \{P\in {\mathcal K}  |  x\not\in \sigma(P)\}.
\end{array}\,$$
   \end{prop}
\noindent
In other words, constructing the spectrum of ${\mathcal K}$ as above is the best way of associating a topological space with a system of supports to ${\mathcal K}$. This has the following important consequence.
   \begin{prop} {\em (Agreement, \cite[Corollary $5.6$]{spc})}\label{agreement} \\
   The underlying topological space of a noetherian scheme $X$ is homeomorphic to $\spc (D^{\perf}(X))$ via the map
   $$\begin{array}{rl} 
 x\mapsto \{P\in D^{\perf}(X)  |  x\not\in \supph(P)\}.
\end{array}\,$$
   \end{prop} 
  

\newpage
\section{The stable homotopy category}
\noindent
Another prominent example of a tensor triangulated category is the stable homotopy category. 
Objects of this category are {\em spectra}, i.e. sequences $P=\{P_i\}_{i\in {\mathbb Z}}$ of CW-complexes together with connecting maps $\Sigma P_i\to P_{i+1}$, such that the adjoints $P_i\to \Omega P_{i+1}$ are weak homotopy equivalences. The morphisms are stable homotopy classes of maps.
Shift of indices gives an autoequivalence, and mapping sequences induce the exact triangles. A compatible monoidal structure is inherited from the smash product of pointed spaces (cf. \cite[Example $1.2.3$]{asht}).  We shall be concerned with the full subcategory ${\mathcal S}_{0}$ of compact objects in the stable homotopy category. We refer to ${\mathcal S}_{0}$ as the {\em finite stable homotopy category}, and to its objects as {\em finite spectra}. \\
 \\
    \subsection{Classification of thick subcategories} \label{classification}
For a fixed prime number $p$, one can consider the {\em $p$-local stable homotopy category}. We shall consider its full subcategory ${\mathcal S}_{0}^p$ of compact objects. The objects of this category are finite spectra whose homotopy groups are $p$-local, i.e. finite spectra $P$ such that $\pi_{\ast}(P)\simeq \pi_{\ast}(P)\otimes {\mathbb Z}_{(p)}$. \\
\\
For a fixed prime number $p$, there is a spectrum $K(n)$ for each integer $n\geq 1$, called the {\em $n$th Morava K-theory (at $p$)}. It has the following properties:
$$\begin{array}{rl} 
(M1) &\mbox{The coefficient ring $K(n)_{\ast}$ is isomorphic to ${\mathbb F}_{p}[v_n,v_n^{-1}]$, with $|v_n|=2(p^n -1)$.}\\
(M2) &\mbox{K\"unneth isomorphism: $K(n)_{\ast}(P\wedge Q)\simeq K(n)_{\ast}(P)\otimes_{K(n)_{\ast}} K(n)_{\ast}(Q)$.}\\
(M3) &\mbox{$K(n+1)_{\ast}(P)=0\Rightarrow K(n)_{\ast}(P)=0$ for all $n$.}\\
(M4) &\mbox{If $P\neq 0$ and finite, there exists an $N$ such that $K(n)_{\ast}(P)\neq 0$ for all $n\geq N$.}
\end{array}\,$$
In addition, we define 
$$\begin{array}{rl} 
K(0)=\eilenberg {\mathbb Q},
\end{array}\,$$
to be the rational Eilenberg-Mac Lane spectrum, and
$$\begin{array}{rl} 
K(\infty)=\eilenberg {\mathbb F}_p,
\end{array}\,$$
to be the mod $p$ Eilenberg-Mac Lane spectrum.\\
 \\
The Morava K-theories determine thick subcategories of ${\mathcal S}_{0}^{p}$. Namely, for any integer $n\geq 1$, the full subcategory 
$$\begin{array}{rl} 
{\mathcal S}_n^{p}:=\{P\in {\mathcal S}_{0}^{p}  |  K(n-1)_{\ast}(P)=0\},
\end{array}\,$$
of $(n-1)$-acyclic $p$-local spectra is thick. 
 The fact that the converse holds is a deep theorem in topology. 
\begin{theor} {\em (Hopkins-Smith, \cite[Theorem $7$]{nilpotence} )}\label{nilpotence}\\
A subcategory ${\mathcal A}$ of ${\mathcal S}_{0}^{p}$ is thick if and only if ${\mathcal A}={\mathcal S}_n^{p}$ for some $n\geq 1$. Further, these subcategories form  a nested filtration of ${\mathcal S}_{0}^{p}$:
$$\begin{array}{rl} 
{\mathcal S}_{\infty}\subset\cdots \subset {\mathcal S}_{n}^{p}\subset {\mathcal S}_{n-1}^{p}\subset\cdots\subset {\mathcal S}_{1}^{p}\subset {\mathcal S}_{0}^{p},
\end{array}\,$$
where all inclusions are proper.
\end{theor}
\newpage
\noindent
 Using the thick subcategory classification for $p$-local finite spectra, S. K. Chebolu derives the following classification of thick subcategories of the finite stable homotopy category ${\mathcal S}_0$.
 \begin{theor} {\em (\cite[Theorem $7.3$]{krull})}\label{chebolus}\\
 A subcategory ${\mathcal A}$ of ${\mathcal S}_{0}$ is thick if and only if ${\mathcal A}={\mathcal S}_{0}$, or
 $$\begin{array}{rl} 
{\mathcal A}=\coprod_{p\in S} {\mathcal S}_{n_p}^{p},
\end{array}\,$$
for some set $S$ of prime numbers, and some integers $n_p\geq 1$.\\
 \\
 \end{theor}
\section{The spectrum of ${\mathcal S}_{0}$}
\noindent
In this section, we apply the functor $\spc(-)$ to the finite stable homotopy category ${\mathcal S}_{0}$, and to its $p$-local versions ${\mathcal S}^{p}_{0}$. In addition to the classification results of the previous section, we will need the classification of primes carried out below in Theorem \ref{stolt} in order to describe the resulting topological spaces.\\
\begin{lem} \label{key} If $p$ and $q$ are distinct prime numbers, then for all integers $i\geq 1$, and all integers $j\geq 0$
$$\begin{array}{rl} 
P\in {\mathcal S}_{i}^{p}\cap {\mathcal S}_{j}^{q} \Rightarrow P\sim\ast.
\end{array}\,$$
\end{lem}
\begin{proof}
If a spectrum $P$ is both $p$-local and $q$-local, then $P$ is $r$-local for any prime number $r$. In other words, $P$ is rational. By the inclusions of Theorem \ref{nilpotence}, $P\in {\mathcal S}_{1}^{p}$, meaning that $P$ has vanishing rational homology. Hence $P$ is contractible. 
\end{proof}
\begin{theor} {\em (Classification of primes)}\label{stolt}\\
$(i)$ Every full triangulated subcategory ${\mathcal A}\subset{\mathcal S}_{0}$ satisfies
 $$\begin{array}{rl} 
Q\in {\mathcal A} \mbox{  and  } P\in {\mathcal S}_0 \Rightarrow  P\wedge Q\in   {\mathcal A}.
\end{array}\,$$
In particular, every thick subcategory of ${\mathcal S}_{0}$ is $\otimes$-thick.\\
$(ii)$ Every thick subcategory of ${\mathcal S}^{p}_{0}$ is $\otimes$-thick.\\
$(iii)$ The primes of ${\mathcal S}^{p}_{0}$ are precisely the subcategories ${\mathcal S}_{n}^{p}$, for $n\geq 1$.\\
$(iv)$ The subcategory
 $$\begin{array}{rl} 
{\mathcal M}=\coprod_{p} {\mathcal S}_{1}^{p},
\end{array}\,$$
is the only prime of ${\mathcal S}_{0}$.
\end{theor}
\begin{proof}
$(i)$ The triangulated category ${\mathcal S}_{0}$ is generated by the sphere spectrum $S$, so any object $P\in {\mathcal S}_{0}$ can be obtained from $S$ by applying successively the autoequivalence $\Sigma$ and taking cofibers of maps. Since the smash product is compatible with these operations, showing that $P\wedge Q\in {\mathcal A}$, for $Q$ in a full triangulated subcategory ${\mathcal A}$, eventually reduces to showing that $S\wedge Q\in {\mathcal A}$. But $S$ is a unit for the smash product, so this is obvious.\\
$(ii)$ follows from $(i)$.\\
$(iii)$ follows from Hopkins' and Smith's theorem (Theorem \ref{nilpotence}) and $(ii)$, together with the K\"unneth isomorphism for Morava K-theory. \\
$(iv)$ To see that ${\mathcal M}$ is a prime, assume that $P\wedge Q\in {\mathcal M}$ and $P\not\in {\mathcal M}$. The subcategory
 $$\begin{array}{rl} 
{\mathcal A}:=\{R\in {\mathcal S}_0  |  R\wedge Q\in {\mathcal M} \}
\end{array}\,$$
is thick, and contains ${\mathcal M}$ properly, since $P\in {\mathcal A}\setminus {\mathcal M}$. By Chebolu's theorem (Theorem \ref{chebolus}), ${\mathcal M}$ is maximal for inclusion of proper thick subcategories, so ${\mathcal A}$ must equal ${\mathcal S}_0$. In particular, the unit $S$ for the smash product belongs to ${\mathcal A}$, and hence $Q\in {\mathcal M}$. 
\newpage
\noindent
Now we show that no other $\otimes$-thick subcategories are primes. Let ${\mathcal A}$ be a $\otimes$-thick proper subcategory of ${\mathcal S}_{0}$. Since we are assuming that ${\mathcal A}\neq {\mathcal M}$, we can use Theorem \ref{chebolus} to conclude that there exists a prime number $p$ and a spectrum $P\in {\mathcal S}_1^{p}$ such that $P\not\in {\mathcal A}$. Let $q$ be any prime number different from $p$, and choose a spectrum $Q\in {\mathcal S}_{0}^{q}\setminus {\mathcal S}_{1}^{q}$. (This is possible by Theorem \ref{nilpotence}.) Then $Q\not\in {\mathcal A}$ by Theorem \ref{chebolus}. However $P\wedge Q\in {\mathcal S}_1^{p}\cap {\mathcal S}_0^{q}$ by $(i)$, so by Lemma \ref{key}, $P\wedge Q\sim\ast$. In particular $P\wedge Q\in {\mathcal A}$, and so ${\mathcal A}$ is not a prime.
\end{proof}
\begin{corol} \label{isapt} As sets
$$\begin{array}{rl} 
\spc ({\mathcal S}_{0})&\simeq \ast, \mbox{  and}\\
\spc ({\mathcal S}^{p}_{0})&\simeq {\mathbb N}.
\end{array}\,$$
\end{corol}
 \noindent
 While the topological space $\spc ({\mathcal S}_{0})$ is trivial, the topological space, $\spc ({\mathcal S}^{p}_{0})$ is more interesting. Definition \ref{thespectrum} implies that any closed subset $Z$ of $\spc({\mathcal S}^{p}_0)$ has the property
 $$\begin{array}{rl} 
S\in Z \mbox{ and } T\subset S \Rightarrow T\in Z.
\end{array}\,$$
The descending sequence
$$\begin{array}{rl} 
{\mathcal S}_{\infty}\subset\cdots \subset {\mathcal S}_{n+1}^{p}\subset {\mathcal S}_{n}^{p}\subset {\mathcal S}_{n-1}^{p}\subset\cdots\subset {\mathcal S}_{2}^{p}\subset {\mathcal S}_{1}^{p}
\end{array}\,,$$
of primes, implies that ${\mathcal S}_{\infty}=Z({\mathcal S_{0}\setminus {\mathcal S}_{\infty}})$ is the only closed point, and that
$$\begin{array}{rl} 
{\mathcal S}^{p}_{1}=\spc ({\mathcal S}^{p}_{0})\setminus Z({\mathcal S}^{p}_{1}\setminus {\mathcal S}^{p}_{2}),
\end{array}\,$$
is the only open point.\\
 \\
\begin{rema} \label{f1}
The category ${\mathcal S}_{0}$ has been suggested as a candidate for a category of perfect complexes over $\Specfun$. One of the motivations for this is that the singular chain complex functor 
$$\begin{array}{rl} 
\sing: {\mathcal S}^{+}\rightarrow D^{+}({\mathcal O}_{\SpecZ}-mod)
\end{array}\,$$
gives a "lift" of the stable homotopy category of bounded below spectra ${\mathcal S}^{+}$ to the bounded below derived category of ${\mathcal O}_{\SpecZ}$-modules, sending finite spectra to perfect complexes. Modulo the Dold-Kan correspondence $\sing$ is the free abelian group functor ${\mathbb Z}[-]$, which should be used for lifting an ${\mathbb F}_1$-scheme to an ordinary (${\mathbb Z}$-)scheme (cf. \cite{deitmar}). Inspired by this observation, A. Salch defines the morphism 
$$\begin{array}{rl} 
f:\SpecZ\to \Specfun
\end{array}\,$$
to be the adjunction $\sing\dashv \eilenberg$, where $\eilenberg$ is the functor sending a complex to its Eilenberg-MacLane spectrum (cf. \cite{salch}). Using the language of axiomatic stable homotopy theory, this adjunction is a {\em geometric morphism} (cf. \cite[$3.4$]{asht}). \\
 \\
Corollary \ref{isapt} supports the point of view that the finite stable homotopy category ${\mathcal S}_0$ could be a category of perfect complexes over $\Specfun$, since it says that the universal locus for supports of objects in ${\mathcal S}_{0}$ is $\Specfun$.
\end{rema}
\newpage
 \section{Analogies with non-noetherian rings}
 \noindent
 The fact that there are infinite chains of thick subcategories of ${\mathcal S}_{0}$ makes ${\mathcal S}_{0}$ resemble a non-noetherian ring. In the appendix of \cite{neeman}, M. B\"okstedt gives a way of identifying the closed sets of $\spec  R$ for an arbitrary (not necessarily noetherian) ring $R$ via its corresponding triangulated category of complexes. In this section, we apply his method to the categories ${\mathcal S}_0$ and ${\mathcal S}^{p}_0$.\\
\begin{defi} \label{filtering}
Let ${\mathcal K}$ be a triangulated category. A thick subcategory ${\mathcal A}\subset {\mathcal K}$ is {\em principal} if it is the triangulated subcategory generated by an element $P\in {\mathcal A}$. If this is the case, we use the notation ${\mathcal A}=\langle P\rangle$. We denote the set of principal subcategories of ${\mathcal K}$ by $PS({\mathcal K})$.\\
 \\
A subset ${\mathcal F}\subset PS({\mathcal K})$ is said to be {\em filtering} if \\
 $(FS1)$ $\langle Q\rangle\in {\mathcal F}$ and $\langle Q\rangle \subset \langle P\rangle$ implies $\langle P\rangle\in {\mathcal F}$ for all $\langle P\rangle\in PS({\mathcal K})$, and\\
 $(FS2)$ for all $\langle P\rangle, \langle Q\rangle\in {\mathcal F}$, there exists a $\langle R\rangle \in {\mathcal F}$ such that $\langle R\rangle \subset \langle P\rangle$ and $\langle R\rangle \subset \langle Q\rangle$.
\end{defi}
\noindent
According to \cite[Proposition $A.7$]{neeman}, closed subsets of $\spec  R$ are in bijection with filtering subsets of $PS(D^{\perf}(\spec  R))$ via the maps
$$\begin{array}{rl} 
Y\mapsto \{\langle P\rangle  |  Y\subset \supph(P)\},
\end{array}\,$$
and
$$\begin{array}{rl} 
{\mathcal F}\mapsto \bigcap_{\langle P\rangle\in {\mathcal F}}\supph (P).
\end{array}\,$$
 \\
 After stating this, M. B\"okstedt invites the reader to see what it gives for ${\mathcal S}_0$. 
 \begin{theor}{\em (Classification of filtering subsets)}\label{stolt2}\\
 $(i)$ The principal thick subcategories of ${\mathcal S}_0$ are precisely ${\mathcal S}_0$ and ${\mathcal S}_{n}^p$, for a prime number $p$ and an integer $n\geq 1$.\\
  $(ii)$ The filtering subsets of $PS({\mathcal S}_0)$ are precisely $PS({\mathcal S}_0)$, $\{{\mathcal S}_0\}$, and the subsets of the form
  $$\begin{array}{rl} 
\{{\mathcal S}_{n}^p\}_{n\leq N}\cup \{{\mathcal S}_0\},
\end{array}\,$$
for a prime number $p$, and an integer $N\geq 1$.\\
$(iii)$ All thick subcategories of ${\mathcal S}^{p}_0$ are principal.\\
  $(iv)$ The filtering subsets of $PS({\mathcal S}^{p}_0)$ are precisely $PS({\mathcal S}^{p}_0)$, and the subsets of the form
  $$\begin{array}{rl} 
\{{\mathcal S}_{n}^p\}_{n\leq N},
\end{array}\,$$
for an integer $N\geq 0$.
 \end{theor}
 \begin{proof}
 $(i)$ Note first that ${\mathcal S}_0$ and the ${\mathcal S}_{n}^p$'s are principal: ${\mathcal S}_0$ is generated by the sphere spectrum, and the strict inclusions of Theorem \ref{nilpotence} imply that for all $n$, ${\mathcal S}_{n}^p$ is generated by any element of ${\mathcal S}_{n}^p\setminus {\mathcal S}_{n+1}^p$. Using the classification in Theorem \ref{chebolus}, let ${\mathcal A}=\coprod_{p\in S} {\mathcal S}_{n_p}^{p}$ be a thick subcategory. Assuming that ${\mathcal A}=\langle P \rangle$ is principal, implies that there is a $p$ and an $n$ such that the generator $P$ belongs to ${\mathcal S}_{n}^{p}$. Since ${\mathcal S}_{n}^{p}$ is thick, it follows that ${\mathcal A}={\mathcal S}_{n}^{p}$.\\
 $(ii)$ Clearly, the subsets in the statement satisfy $(FS1)$ and $(FS2)$ in Definition \ref{filtering}. By $(FS1)$, a filtering subset containing ${\mathcal S}_{N}^p$ must also contain ${\mathcal S}_{n}^p$ for all $n\leq N$. It remains to show that for $p\neq q$, no filtering subset can contain both ${\mathcal S}_{k}^p$ and ${\mathcal S}_{l}^q$. By $(PS2)$, this implies the existence of integers $i$ and $j$ such that ${\mathcal S}_{i}^p={\mathcal S}_{j}^q$. By Lemma \ref{key}, ${\mathcal S}_{i}^p$ must then consist entirely of contractible spectra, which contradicts property $(M4)$ of Morava K-theory. \\
 \newpage
 \noindent
 $(iii)$ follows by applying the argument in the proof of $(i)$ to ${\mathcal S}^{p}_0$.\\
 $(iv)$ follows directly from $(iii)$ together with the axioms $(FS1)$ and $(FS2)$.
 \end{proof}
 \noindent
 Note that filtering subsets of $PS({\mathcal S}_0)$ are {\em not} in bijection with closed sets of $\spc({\mathcal S}_0)$, since there are infinitely many filtering subsets, but only one closed set. However, the filtering subsets of $PS({\mathcal S}^{p}_0)$ are precisely the open sets of $\spc({\mathcal S}^{p}_0)$, so in this case they are certainly in bijection with the closed sets. \\
 \\
\section*{Acknowledgments}
\noindent
The author would like to thank Ryszard Nest, Jesper Grodal, and Arvo P\"art. The author was supported by Marie Curie grant no.
MRTN-CT-2006-031962, and by the Danish National Research Foundation through the Centre for Symmetry and Deformation.\\
 \\

\end{document}